\documentclass[10pt,reqno]{amsart}
\usepackage{latexsym, amsmath, amssymb, longtable, booktabs}
\usepackage{amsthm}

\usepackage[matrix,arrow]{xy}

\begin{document}

\title[Modular Brauer Classes]{On the Brauer class of Modular Endomorphism
Algebras}

\author[E. Ghate]{Eknath Ghate} 
\address{
School of Mathematics \\
Tata Institute of Fundamental Research \\
Homi Bhabha Road, Mumbai  400 005,  India.
}
\email{eghate@math.tifr.res.in}

\author[E. Gonz\'alez-Jim\'enez]{Enrique Gonz\'alez-Jim\'enez} 
\address{
Departamento de Matem\'aticas \\
Facultad de Ciencias \\
Universidad de Alcal\'a \\
E-28871 Madrid \\
Spain.
}
\email{enrique.gonzalezj@uah.es}

\author[J. Quer]{Jordi Quer}
\address{Departament Matem\`atica Aplicada II \\          
 Universitat Polit\`ecnica de Catalunya   \\
 Jordi Girona, 1-3   \\ 
 Campus Nord, Edifici Omega, Despatx 438  \\                   |
 08034 Barcelona, Spain                
}
\email{Jordi.Quer@upc.es}

\date{February 5, 2005}

\begin{abstract}
We investigate the Brauer class of the endomorphism
algebra of the motive attached to a non-CM form.  
The ramification of the algebra is shown in many
cases to be controlled by the normalized slopes of the form.
\end{abstract}
\maketitle

\vspace{-0.6cm}
%\tableofcontents

\newcommand\A{\mathbb{A}}
\newcommand\C{\mathbb{C}}
\newcommand\F{\mathbb{F}}
\newcommand\Ham{\mathbb{H}}
\newcommand\HH{{\mathrm H}}
\newcommand\mM{{\mathrm M}}
\newcommand\Q{\mathbb{Q}}
\newcommand\Qbar{{\overline{\Q}}}
\newcommand\sQ{{\mathcal{Q}}}
\newcommand\R{\mathbb{R}}
\newcommand\Z{{\mathbb{Z}}}
\newcommand\gp{\mathfrak{p}}
\newcommand\mZ{{\mathrm Z}}
\newcommand\tensor{{\otimes}}
\newcommand\disc{{\mathrm{disc}}} 
\newcommand\divides{{\> \big| \>}}
\newcommand\notdivides{{\> \not\bigl| \>}}
\newcommand{\eps}{\epsilon}
\newcommand\Aut{{\mathrm {Aut}}}
\newcommand\Br{{\mathrm {Br}}}
\newcommand\crys{{\mathrm{crys}}} 
\newcommand\End{{\mathrm {End}}}
\newcommand\Fil{{\mathrm {Fil}}}
\newcommand\Frob{{\mathrm {Frob}}}
\newcommand{\GL}{\operatorname{GL}}
\newcommand\Hom{{\mathrm {Hom}}}
\newcommand\Inf{{\mathrm {Inf}}}
\newcommand\isom{{\> \cong \>}}
\newcommand{\isoTo}{\buildrel \sim \over \longrightarrow}
\newcommand\Gal{{\mathrm {Gal}}}
\newcommand{\Leps}{L_\varepsilon}
\newcommand\lcm{{\mathrm {lcm}}}
\newcommand\mat{{ \left( {a \atop c} {b \atop d} \right)}}
\newcommand\modulo{{\mathrm {mod} \> }}
\newcommand{\New}{{{New}}}
\newcommand{\Norm}{\operatorname{Norm}}
\newcommand\ord{{\mathrm {ord}}}
\newcommand{\ssigma}{{{ }^\sigma \!}}
\newcommand{\Sch}{\operatorname{Sch}}
\newcommand{\Sh}{\operatorname{Sh}}
\newcommand{\SL}{\operatorname{SL}}
\newcommand{\To}{\longrightarrow}
\newcommand\Tr{{\mathrm {Tr}}}
\newcommand{\ur}{\mathrm{ur}}
\newcommand{\legendre}[2] {\left(\frac{#1}{#2}\right)}
\newcommand{\op}{\operatorname}

\numberwithin{equation}{section}
\newtheorem{thm}{Theorem}[section]
\newtheorem{prop}[thm]{Proposition}
\newtheorem{cor}[thm]{Corollary}
\newtheorem{remark}[thm]{Remark}
\numberwithin{table}{section}

%\numberwithin{equation}{subsection}

%\newtheorem{cor}[equation]{Corollary}
%\newtheorem{prop}[equation]{Proposition}
%\newtheorem{lemma}[equation]{Lemma}
%\newtheorem{theorem}[equation]{Theorem}
%\theoremstyle{definition}
%\newtheorem{definition}[equation]{Definition}
%\theoremstyle{remark}
%\newtheorem{remark}[equation]{Remark}
%\newtheorem{example}[equation]{Example}
%\newtheorem{claim}[equation]{Claim}

\section{Introduction}

In this paper we study the Brauer class of the 
endomorphism algebra $X_f$ of the motive attached to a primitive 
elliptic modular cusp form $f$ without complex multiplication. Our study
includes the case of forms of weight 2, where the associated 
motive is an abelian variety. 

It is a fundamental fact that $X_f$ has a natural crossed product structure. 
This was proved by Ribet \cite{Ribet80} and Momose \cite{Momose81} 
in the case of weight 2, and extended to forms of higher weight
in \cite{Brown-Ghate03} subject to an injectivity constraint, which we  
remove here. It follows that $X_f$ is a central simple algebra over
a subfield $F$ of the Hecke field of $f$. Moreover $X_f$ is 2-torsion 
when considered as an element of the Brauer group of $F$. 
Thus $X_f$ is isomorphic to a matrix algebra over $F$, or 
a matrix algebra over a quaternion division algebra over $F$. 
Ribet has remarked in \cite{Ribet80} that it seems difficult
to distinguish these cases by pure thought. His remark pertains to the
case of weight 2, but is equally relevant in higher weight.
The chief motivation of this paper (and to a large extent \cite{Brown-Ghate03}) 
is to give as complete a picture as possible of the Brauer class of $X_f$. 

In recent years the notion of slope has played an important role in the 
theory of elliptic modular forms. For instance this notion is fundamental in 
parameterizing families of elliptic modular cusp forms, as in the work of 
Hida (slope 0) and Gouv$\hat{\mathrm e}$a, Mazur, Coleman (finite slope). 
Remarkably, the notion of slope turns out to be useful in studying the 
Brauer class of modular endomorphism algebras as well. In fact our main result 
is that at a finite place of $F$ not dividing the level of $f$ the ramification
of $X_f$ is essentially completely determined by the {\it parity} of the (normalized) slope 
of $f$ when this slope is finite. 

At places of $F$ dividing the level  
our knowledge of the ramification of $X_f$ is less complete. However we show that
it is still governed to some extent by the slopes
of $f$, at least at certain places where the underlying local 
representation is in the principal series. On the other hand at places for which
the local representation is of Steinberg type one knows quite a bit about
the ramification of $X_f$ (see \cite{Brown-Ghate03}). 
Predicting the ramification of $X_f$ at the remaining bad places 
(the supercuspidal places) is still an open problem.

We end this paper with tables of the Brauer 
class of $X_f$ for all forms $f$ of small weight and level 
with $F = \Q$.

\section{Statement of results}
  \label{statements}

We now give more precise statements of our results.
Let $f = \sum a_n q^n$ be a primitive cusp form of weight $k \geq 2$, 
level $N \geq 1$ and nebentypus $\epsilon$. Here primitive means that $f$ is 
a normalized newform that is a common eigenform of all the Hecke operators. 
Let $M_f$ denote the abelian variety associated to $f$ as constructed by Shimura when 
$k = 2$, and let $M_f$ denote the Grothendieck motive attached to $f$ constructed by Scholl 
in \cite{Scholl90} when $k > 2$.  
Let $\End(M_f)$ be the ring of endomorphisms of $M_f$ defined over $\Qbar$. When
$k > 2$ we work modulo cohomological equivalence. Set $X_f := \End(M_f) \otimes \Q$. 

The first result is that $X_f$ has a natural structure of a crossed
product algebra. To state this result more precisely we need some 
notation. Let $E = \Q(a_n)$ denote the Hecke field of $f$. 
Then $E$ is either a totally real or a 
CM number field. Assume from now on that $f$ does not have
complex multiplication. A pair $(\gamma, \chi_\gamma)$ where $\gamma \in \Aut(E)$ and
$\chi_\gamma$ is an $E$-valued Dirichlet character is said to be an extra
twist for $f$ if $a_p^\gamma = a_p \cdot \chi_\gamma(p)$ for all but finitely many
primes $p$. Let $\Gamma$ denote the set of $\gamma \in \Aut(E)$ 
such that $f$ has a twist by $(\gamma, \chi_\gamma)$ for some 
$E$-valued Dirichlet character $\chi_\gamma$. In turns out that
$\Gamma$ is an abelian subgroup of $\Aut(E)$.
For $\gamma$, $\delta \in \Gamma$ set
\begin{eqnarray}
  \label{Jacobi-sums}
  c(\gamma, \delta) = \frac{G(\chi_\gamma^{-1}) G(\chi_\delta^{-\gamma})}
                           {G(\chi_{\gamma \delta}^{-1})}
\end{eqnarray}
where $G(\chi)$ is the Gauss sum of the primitive Dirichlet character
associated to $\chi$. Then $c \in \mathrm{Z}^2(\Gamma, E^\times)$ is a
2-cocycle which turns out to be $E^\times$-valued. 
Let $X$ denote the crossed product algebra associated to $c$.
For the reader's convenience we recall the definition of $X$.
For each $\gamma \in \Gamma$ let $x_\gamma$ denote
a formal symbol. Then as an $E$-vector space $X$ is finite dimensional
with basis given by the symbols $x_\gamma$:  
\begin{eqnarray}
  X = \bigoplus_{\gamma \in \Gamma} E \> x_\gamma,
\end{eqnarray}
and as an algebra $X$ has structure given by the relations
\begin{eqnarray}
  x_\gamma \cdot e &  = & \gamma(e) \, x_\gamma \\
  x_\gamma \cdot x_\delta & = & c(\gamma, \delta) \, x_{\gamma \delta},  \nonumber
\end{eqnarray}
where $e \in E$ and $\gamma$, $\delta \in \Gamma$. 

If $k = 2$ then it is a result of Ribet \cite[Theorem 5.1]{Ribet80} 
and Momose \cite[Theorem 4.1]{Momose81} that $X_f$ is isomorphic to $X$.
On the other hand if $k > 2$ then $X_f$ contains a 
sub-algebra isomorphic to $X$  
(see \cite{Momose81} and \cite[Theorem 1.0.1]{Brown-Ghate03}).
Here, building on the above results, we prove the following.

\begin{thm}
  \label{structure}
  Let $k \geq 2$.
  Then $X_f$ is isomorphic to $X$. 
\end{thm}

Let $F$ be the number field contained in $E$ which is the fixed field of $\Gamma$. 
%Theorem~\ref{structure} implies that 
Then $X$ is isomorphic to a central simple algebra over $F$
which is easily seen to be 2-torsion when considered as an element of the 
Brauer group of $F$. As a result $X$ is either a matrix algebra over $F$ or a matrix
algebra over a quaternion division algebra over $F$. We wish 
to distinguish these cases.

Recall that by global class field theory there is an injection 
\begin{eqnarray} 
  \Br(F) \hookrightarrow \oplus_v \Br(F_v)
\end{eqnarray} 
where $v$ runs through the places of $F$ and $F_v$ is the completion of $F$ at $v$.
Thus to study the Brauer class of $X$ it suffices to study its image 
$X_v = X \otimes_F F_v$ for each place $v$ of $F$ under the above map. 
Since $X$ is $2$-torsion in the Brauer group of $F$, the algebra $X_v$ is 
{\it a fortiori } either a matrix algebra over $F_v$ or a 
matrix algebra over a quaternion division algebra over $F_v$.
 
As far as the infinite places are concerned the field $F$ is easily seen to be totally real,
and it follows from a result of Momose \cite[Theorem 3.1 ii)]{Momose81} 
(see also \cite[Theorem 3.1.1]{Brown-Ghate03}) that $X$ is totally indefinite 
if $k$ is even or totally definite if $k$ is odd. 

Now suppose that $v$ is a finite place of $F$ of residue characteristic $p$ with $p$
coprime to $N$. Then $a_p^2 \epsilon(p)^{-1} \in F$. Set 
\begin{eqnarray}
  m_v & := & [F_v : \Q_p] \cdot v(a_p^2 \epsilon(p)^{-1}) \> \in \> \Z \cup \{\infty\}
\end{eqnarray}
where $v$ is normalized so that $v(p) = 1$. We shall show that the 
structure of $X_v$ is essentially determined by the {\it parity} of $m_v$ when it is finite. 
If $w$ is a place of $E$ lying over $v$ then $w(a_p) = \frac{1}{2}v(a_p^2 \epsilon(p)^{-1})$ is 
called the {\it slope} of $f$ at $p$ (with respect to $w$). Thus we show that there is a close 
connection between the ramification of $X_v$ away from the level and the parity of 
(normalized) slopes. More precisely we have the following theorem. 

\begin{thm}
  \label{parity} 
  Let $p$ be a prime such that
  \begin{itemize}
    \item $p$ does not divide $N$, and, 
    \item $p \neq 2$ if $F \neq \Q$.
  \end{itemize}
  Also assume $a_p \neq 0$. 
  Let $v$ be a place of $F$ lying over $p$. Then
  \begin{center} 
    $X_v$ is a matrix algebra over $F_v$ if and only if $m_v \in \Z$ is even,
  \end{center} 
  except possibly in the exceptional case that $p$ splits in all
  the quadratic fields cut out by the extra twists of $f$ 
  in which case $X_v$ is necessarily
  a matrix algebra over $F_v$. 
\end{thm}

A result of this kind was proved in \cite[Theorem 1.0.4]{Brown-Ghate03} for cusp 
forms having quadratic extra twists, or equivalently, real nebentypus character. 
That such a result might be true in general became clear after extensive
numerical computations were made by the second author. The proof of the general case 
combines ideas from \cite{Quer98} and \cite{Brown-Ghate03}. 

If $f$ is ordinary at $v$ (that is $v(a_p^2 \epsilon(p)^{-1}) = 0$) then $m_v = 0$, 
and it follows from the theorem that $X_v$ is a matrix algebra over $F_v$. 
This was already known even if $p = 2$ (see \cite[Theorem 6]{Ribet81} for $k = 2$ and 
\cite[Theorem 3.3.1]{Brown-Ghate03} for $k > 2$). Thus the theorem above
may be considered as a generalization of these results.   

For a more detailed explanation of the exceptional case mentioned in
the statement of the theorem the reader is referred to Theorem~\ref{supersingular} below. 
It is also possible to deal with the case $a_p = 0$ (and $p$ still coprime to $N$). In this case 
$m_v$ blows up but it may easily be substituted for by a closely related (finite) integer 
(see Proposition~\ref{a_p=0}). We point out here that the above mentioned results 
imply that $X$ can only be 
ramified at the primes dividing 
\begin{eqnarray}
  2 \cdot N \cdot \disc(E) \cdot \infty
\end{eqnarray}
where $\disc(E)$ is the discriminant of $E$ and $\infty$ is the unique infinite place of 
$\Q$ (Corollary~\ref{where-ramified}).  

The proof of the theorem above is based on an explicit computation of symbols
appearing in a formula for the Brauer class of $X$. A more conceptual approach, based on a 
study of the filtered $(\phi, N)$-modules of Fontaine attached to the local Galois 
representations associated to $f$, is available. This approach has so far been 
more successful for studying the ramification of $X$ only when 
the slope is small compared to the weight 
(see for instance \cite[Theorem 1.0.3]{Brown-Ghate03}). However, it can be 
applied to study the ramification of $X$ at some bad places $v | N$ 
(see \cite[Theorems 1.0.5 and 1.0.6]{Brown-Ghate03}). 
Here we push this approach further and prove the following result.

\begin{thm}
  \label{bad-singular-intro}
  Suppose that $p|N$ and that the power of $p$ dividing $N$ is the same
  as the power of $p$ dividing the conductor of $\epsilon$. 
  Let $v$ be a place of $F$ lying over $p$. Let $\alpha \in \Q$ be such that
  $0 \leq \alpha < (k-1)/2$
  and $\alpha$ has odd denominator. 
  If for each place $w$ of $E$ lying over $v$ either 
  \begin{eqnarray}
    w(a_p) = \alpha & \text{or} &  \bar{w}(a_p) = \alpha 
  \end{eqnarray}
  then $X_v$ is a matrix algebra over $F_v$.
\end{thm}

If the power of $p$ dividing $N$ is larger
than the power of $p$ dividing the conductor of $\epsilon$,  our knowledge
of $X_v$ is less complete, except in the very special case  
when $p||N$ and the conductor of $\epsilon$ is prime to $p$ (the Steinberg case). 
In this last case it turns out that the ramification of $X_v$ is related to the parity
of the weight $k$ of $f$. For more precise statements see the end of section 3 
of \cite{Brown-Ghate03}.

\section{Crossed Product Structure}
  \label{section-structure}

In this section we prove that $X_f$ is isomorphic to
the crossed product algebra $X$ (Theorem~\ref{structure}).
In view of the work of Ribet and Momose we shall assume in 
this section that $k > 2$. 

Let $\ell$ be a prime and let $M_\ell$ denote the $\ell$-adic
realization of $M_f$. Recall that $M_\ell$ is a $\Q_\ell$-vector space 
with an action of $\Gal(\Qbar/\Q)$. For a subgroup
$H$ of $\Gal(\Qbar/\Q)$ let $\End_H(M_\ell)$ denote the endomorphisms
of $M_\ell$ which commute with $H$. An endomorphism of $M_f$
gives rise to an endomorphism of each of its realizations. One
therefore obtains a map 
\begin{eqnarray}
  \label{alpha-modular}
  \alpha : \End(M_f) \otimes \Q_\ell \rightarrow \End_H(M_\ell).
\end{eqnarray}
where $H$ is a sufficiently deep finite index subgroup of 
$\Gal(\Qbar/\Q)$. In \cite{Brown-Ghate03} it was shown that 
$X_f$ contains a sub-algebra generated by certain twisting
operators which is isomorphic to the crossed product algebra $X$. 
Moreover it was shown that if $\alpha$ is injective then $X_f$ is 
isomorphic to $X$ because of
dimension considerations. Indeed, by 
\cite[Theorem 4.4]{Ribet80} (whose proof carries over to the
case $k > 2$), the $\Q_\ell$-dimension of $\End_H(M_\ell)$ is $[E:F][E:\Q]$, 
which is also the $\Q_\ell$-dimension of $X \otimes \Q_\ell$.

So it suffices to prove that $\alpha$ is injective.
To do this we work slightly more generally. 
Let $X$ be a smooth irreducible projective variety over $\Q$ of dimension $d$. 
Let $Z(X \times X)$ be the rational vector space generated by the 
irreducible sub-varieties of $X \times X$ over $\Qbar$ of codimension $d$. 
Fix an embedding $\Qbar \hookrightarrow \C$. Let 
$\HH_B^{2d}(X\times X _{/\C})(d)$ denote Betti cohomology (with 
coefficients $(2 \pi i)^d \Q$) and let 
$c_B : Z(X \times X) \rightarrow \HH_B^{2d}(X\times X _{/\C})(d)$ 
be the cycle class map. 
Let $Z_h(X\times X)$ be the quotient
\begin{equation}
Z_h(X\times X) \> :=  \> Z(X\times X)/ \ker(c_B) \> = \> Z(X\times X)/\sim
\end{equation} 
where $\sim$ is the cohomological equivalence relation. 
Thus for $Z \in Z(X\times X)$ one has $Z \sim 0$ if and only if the 
image of $Z$ in $\HH_B^{2d}(X\times X _{/\C})(d)$ under $c_B$ is zero.
Recall that $Z_h(X\times X)$ has a natural ring structure where 
multiplication is induced by the composition product of correspondences. 
Let $p \in Z_h(X \times X)$ be a projector.
Let $M = (X, p)$ be a motive, where $X$ and $p$ are as above. 
Recall that by definition 
\begin{eqnarray}
  \End(M) := \frac{\{ Z  \in Z_h(X\times X) : Z\circ p = p \circ Z \}}
                    {\{ Z\in Z_h(X\times X) : Z\circ p = p \circ Z =0 \} }.
\end{eqnarray}
We show that in this setting the natural map
\begin{eqnarray}
  \label{alpha-general}
  \alpha : \End(M) \otimes \Q_\ell \rightarrow  \End(M_\ell) 
\end{eqnarray}
is injective. Note that $\alpha$ is equivariant for the action of $\Gal(\Qbar/\Q)$ on 
both sides.

Consider the cycle class map 
$c_\ell : Z(X \times X) \rightarrow  \HH_\ell^{2d}(X\times X)(d)$ 
to $\ell$-adic cohomology. There is a comparison isomorphism
\begin{equation}
I_\ell : \HH_B^{2d}(X\times X _{/\C})(d) \otimes \Q_\ell \> \isom \>  
           \HH_\ell^{2d}(X\times X _{/\C})(d) \> \isom \> \HH_\ell^{2d}(X\times X)(d)
\end{equation}
where the first isomorphism is (a twist of) the canonical comparison
isomorphism between Betti and $\ell$-adic cohomology for smooth
projective varieties over $\C$ (see \cite[Theorem 3.12]{Milne80}), and the second 
isomorphism is (again a twist of the one) induced by the embedding 
$\Qbar \hookrightarrow \C$ via the proper base change theorem.   
The two cycle class maps $c_B$ and $c_\ell$ are related via $I_\ell$, that is
$c_\ell =  I_\ell \circ (c_B \otimes 1)$ (see page 21 of \cite{Deligne82} or
page 58 of \cite{Jannsen90}).
It follows that the $\ell$-adic cycle class map factors through the 
cycles (Betti-) cohomologically equivalent to zero, inducing a map
$Z_h(X \times X) \hookrightarrow \HH_\ell^{2d}(X\times X)(d)$. (It also
follows that $Z_h(X \times X)$ and hence $\End(M)$ is defined independently
of the embedding $\Qbar \hookrightarrow \C$ fixed above). 
Since Betti cohomology gives a rational structure on $\ell$-adic cohomology, 
the induced map 
\begin{eqnarray}
  \label{cycle-class-map-for-XxX}
  Z_h(X \times X) \otimes \Q_\ell \hookrightarrow \HH_\ell^{2d}(X\times X)(d)
\end{eqnarray}
continues to be injective. This is the key observation that was missed
in \cite{Brown-Ghate03}. To deduce from this that the map 
$\alpha$ in (\ref{alpha-general}) is also injective is purely formal. We recall the 
argument here for the sake of completeness. By the K\"unneth formula %we have 
\begin{eqnarray}
  \HH_\ell^{2d}(X\times X)(d) 
  = \bigoplus_{q = 0}^{2d} \HH_\ell^{q}(X) \otimes \HH_\ell^{2d-q}(X)(d) 
  = \bigoplus_{q = 0}^{2d} \End(\HH_\ell^{q}(X))
\end{eqnarray}
where the last equality follows since $\HH_\ell^{2d-q}(X)(d)$ is 
dual to $\HH_\ell^{q}(X)$. Let 
\begin{eqnarray}
  \HH_\ell^{q}(X)(p) = \mathrm{Im}\left(p : \HH_\ell^{q}(X) \rightarrow \HH_\ell^{q}(X)\right).
\end{eqnarray} 
By definition $\End(M_\ell) = \oplus_{q = 0}^{2d} \End(\HH_\ell^{q}(X)(p))$. Now, the map
in (\ref{cycle-class-map-for-XxX}) induces a map
\begin{eqnarray}
  \label{cycle-class-map-p}
  \{ Z\in Z_h(X\times X) : Z\circ p = p \circ Z \} \otimes \Q_\ell
  \rightarrow \End(M_l)
\end{eqnarray}
Let $Z = Z \otimes 1 \in Z_h(X\times X) \otimes \Q_\ell$ 
belong to the kernel of (\ref{cycle-class-map-p}). 
Thus we have $Zp = pZ$ and $Z = 0$ in $\End(M_\ell)$. 
Write $Z = Zp +Z(1-p)$ and note that $Z(1-p) \circ p = 0$,
so $Z(1-p)$ acts as 0 on each $\HH_\ell^{q}(X)(p)$. Thus
$Zp = Z-Z(1-p)$ acts as 0 on each $\HH_\ell^{q}(X)(p)$. 
Since $Zp$ clearly acts as 0 on the spaces 
\begin{equation}
\ker(p : \HH_\ell^{q}(X) \rightarrow \HH_\ell^{q}(X))
\end{equation}
we see that
$Zp = 0$ in $\HH_\ell^{2d}(X\times X)(d)$. By the injectivity
of (\ref{cycle-class-map-for-XxX}) we see that $Zp = pZ = 0$ in 
$Z_h(X\times X)$. It follows that the 
kernel of the map (\ref{cycle-class-map-p}) is 
$\{ Z\in Z_h(X\times X) : Z\circ p = p \circ Z = 0\} \otimes \Q_\ell$.
Thus (\ref{cycle-class-map-p}) induces an injective map
$\End(M) \otimes \Q_\ell \hookrightarrow \End(M_\ell)$, which is
precisely the map $\alpha$ in (\ref{alpha-general}). 
This proves Theorem~\ref{structure}.

\section{Ramification and Slopes}
  \label{proof}

Let $X \isom X_f$ denote the endomorphism algebra of $M_f$. 
We now study the relation between the ramification of $X$ and the slopes of $f$.

Let $G = \Gal(\Qbar/\Q)$ and let $G_F = \Gal(\Qbar/F)$. We will sometimes
consider Dirichlet characters as characters of $G$. In particular
$\epsilon$ is a character of $G$. 
For each $g \in G$ let $\sqrt{\epsilon(g)}$ be a 
fixed square-root of $\epsilon(g)$. 
Now, for every extra twist  $(\gamma,\chi_\gamma)$ of $f$, and every pre-image 
of $\gamma$ in $G_F$, which we again denote by $\gamma$,
there is a unique primitive Dirichlet character $\psi_\gamma$ of order 1 or 2
such that $\chi_\gamma(g) = \psi_\gamma(g) \cdot \sqrt{\epsilon(g)}^{\gamma-1}$ for
all $g \in G$ 
(cf. \cite[Lemma 1.5]{Momose81} and \cite[Lemme 2]{Quer98}). 
Let $\Q(\sqrt{t_\gamma})$ be the quadratic field corresponding to $\psi_\gamma$ (in the case
that it has order 2). We assume that $t_\gamma$ is also the discriminant of 
this field.  The characters $\{ \psi_\gamma \divides \gamma \in G_F\}$
form an elementary 2-group. 
Fix once and for all a subset $\Gamma_0 \subset G_F$ such
that $\{ \psi_\gamma \divides \gamma \in \Gamma_0 \}$ is a basis for this group. 
For each $\gamma \in \Gamma_0$
choose square-free positive integers $n_\gamma$ prime to $N$ such that 
$a_{n_\gamma} \neq 0$, and such that for all $\gamma' \in \Gamma_0$,
\begin{eqnarray}
  \label{n_gamma}
  \psi_{{\gamma'}}(n_\gamma) = 
     \begin{cases} 
               -1  & \text{if $\gamma' = \gamma$},  \\
        \>\>\>  1  & \text{if not}.
     \end{cases}
\end{eqnarray}
Also, for a square free integer $n$ which is prime to $N$, set 
$z_n = a_n^2 \epsilon(n)^{-1} \in F$. 

Let $[c_\epsilon]$ denote the 
class of the cocycle $c_\epsilon \in \mZ^2(G_F, \pm1)$ 
defined by
\begin{equation}
c_\epsilon(g, h) = \sqrt{\epsilon(g)} \sqrt{\epsilon(h)}
                          {\sqrt{\epsilon(gh)}}^{-1},
\end{equation}
for $g, h \in G_F$ (see \cite[Section 2]{Quer98}). 

The following result expressing the Brauer class of $X$ in
terms of symbols was proved in \cite[Th\'eor\`eme 3]{Quer98} in 
the case $k = 2$. 

\begin{thm}
  \label{symbols}
  Let $k \geq 2$. Then
  \begin{eqnarray}
    X  & = & [c_\epsilon] \otimes \bigotimes_{\gamma \in \Gamma_0} \big(z_{n_\gamma}, \> t_\gamma \big)  
  \end{eqnarray}
  up to Brauer equivalence.
\end{thm}

\begin{proof}
We make some brief remarks which show that the proof given in \cite{Quer98} for the case 
$k = 2$ continues to hold if $k > 2$. 

Each $g \in G$ acts naturally 
as an automorphism on $X$ and this automorphism fixes $E$ since the elements 
of $E$ are basically Hecke operators and so are defined over $\Q$. 
By the Skolem-Noether theorem the action of $g$ on $X$ must be
given by inner conjugation by some element $e \in X$, which is
well defined modulo $F^\times$. Since $E$ 
is its own commutant in $X$ we see that $e \in E$. 
The association $g \mapsto e$ defines a continuous character 
$\alpha : G \rightarrow E^\times / F^\times$ where the
target has the discrete topology. Write $\tilde\alpha : G \rightarrow E^\times$
for any lift of $\alpha$. 

The first point to note is that the result \cite[Theorem 5.6]{Ribet92}, which says
that the class of $X$ in $\Br(F)$ is cut out by the 2-cocycle 
\begin{eqnarray}
  \label{alpha-cocycle}
  (g, h) \mapsto \frac{\tilde\alpha(g) \tilde\alpha(h)}{\tilde\alpha(gh)}
\end{eqnarray}
for $g, h \in G_F$, continues to hold if $k > 2$. 
Indeed, 
%by Theorem~\ref{structure}, 
the class of $X$ 
in $\Br(F)$ is the same as the class of the cocycle 
$c(g, h) \in \mZ^2(G_F, \bar{F}^\times)$, naturally obtained from the Jacobi sum 
cocycle $c(\gamma, \delta) \in \mZ^2(\Gamma, E^\times)$ in (\ref{Jacobi-sums}) 
by inflation. On the  other hand by \cite[Proposition 1]{Ribet81}, which is 
easily seen to hold for $k > 2$,
this last cocycle cuts out the same class in $\Br(F)$ as the cocycle 
\begin{eqnarray}
  \label{cocycle-chi}
  (g, h) \mapsto \chi_g(h)
\end{eqnarray} 
where $\chi_g = \chi_\gamma$ for the image $\gamma \in \Gamma$ of $g \in G_F$, and $\chi_g$ is thought of 
as a character of $G$. Now using
\cite[Theorem 4]{Ribet81}, which says that $\tilde\alpha(h)^{g - 1} = \chi_g(h)$, 
we see that the cocycle (\ref{cocycle-chi}) defines the same class in 
$\Br(F)$ as the cocycle $(g, h) \mapsto \tilde\alpha(h)^{g - 1}$. 
The proof of \cite[Theorem 4]{Ribet81} uses the Tate conjecture for the abelian
variety attached to $f$. This is proved in \cite[Corollary 1.0.2]{Brown-Ghate03}
for the motive $M_f$ when $k > 2$ subject to an injectivity hypothesis which 
we have removed in Section~\ref{section-structure}. Finally, 
this last cocycle differs from the cocycle (\ref{alpha-cocycle}) by the map 
\begin{eqnarray}
  (g, h) \mapsto \frac{\tilde{\alpha}(h)^g \tilde{\alpha}(g)}{\tilde{\alpha}(gh)}
\end{eqnarray}
which is a coboundary. 

The second point to note is that the cocycle 
(\ref{alpha-cocycle}) is equal to $c_\epsilon \cdot c_d$ up to a coboundary, 
where $c_d$ is the cocycle defined by 
\begin{equation}
c_d(g,h) = \left(\frac{\tilde\alpha(h)}{\sqrt{\epsilon(h)}}\right)^{1-g},
\end{equation}
for $g, h \in G_F$ (see the beginning of Section 2 of \cite{Quer98}). 
The rest of the proof of the theorem, which involves writing
the class of $c_d$ as a product of symbols, proceeds exactly
as in the case $k = 2$ without any change. This proves the theorem.
\end{proof}

\begin{remark} 
A formula similar to that appearing in the theorem above
was proved in \cite[Theorem 4.1.3]{Brown-Ghate03} in the 
case that all the $\chi_\gamma$ are quadratic characters.   
This formula was proved by directly computing 
quadratic Gauss sums. 
\end{remark}

We wish to evaluate the symbols that appear in the expression for
$X$ in Theorem~\ref{symbols}. To do this let us recall some general 
facts about symbols from \cite{Serre79}.
Let $F$ be an arbitrary number field. 
Let $v$ denote a place of $F$ which is either finite or infinite. 
Let $F_v$ denote the completion of 
$F$ at $v$. It is well known that 
\begin{eqnarray}
  \label{local-Brauer-group}
  \Br(F_v) & \isom & \Q/\Z
\end{eqnarray}
if $v$ is finite (and is $\Z/2$ if $v$ is infinite and real and is trivial
if $v$ is infinite and complex). 
Now let $a$ and $b$ be non-zero elements of $F$. Then the symbol
$(a, b)$ determines an element in $\Br(F)[2]$. 
For each finite place $v$ of $F$, let $(a, b)_v$ denote the induced element of 
$\Br(F_v)[2]$. By (\ref{local-Brauer-group}) the symbol
$(a, b)_v$ is completely specified by a sign $+1$ or $-1$. 
This sign can be computed in terms of the $v$-adic valuations of $a$ and $b$.
There are two cases, the tame case: $v \notdivides 2$, and 
the wild case: $v | 2$.

First assume that $v$ is prime to $2$. Fix a uniformizer 
$\pi_v$ of the ring of integers of $F_v$. Write
\begin{eqnarray}
  a & = & \pi_v^{v(a)} \cdot a' \\
  b & = & \pi_v^{v(b)} \cdot b'\nonumber
\end{eqnarray}
where we consider $v$ here to be normalized such that $v(\pi_v)=1$.
In this section $v$ will refer to a valuation which is normalized
in this way unless explicitly stated otherwise.
Then one has
\begin{eqnarray}
  \label{tame}
  (a, b)_v & = & (-1)^{\frac{Nv-1}{2} \> v(a)v(b)} 
                 %(-1)^{(1/2)(Nv-1) \cdot v(a)v(b)}  
                 \cdot \legendre{b'}{v}^{v(a)} 
                 \cdot \legendre{a'}{v}^{v(b)}. 
\end{eqnarray}
Here the symbol $\legendre{c}{v}$ takes the values $\pm1$ and
is $1$ exactly when the image of $c$ is a square in the residue
field at $v$. 

Now assume that $v | 2$. We shall only treat the case 
$F = \Q$ so that $v = 2$. For a unit $u \in \Q_2^\times$ let
$\varepsilon(u)$ denote the residue of $(u-1)/2$ in $\Z/2$ and let 
$\omega(u)$ denote the residue of $(u^2-1)/8$ in $\Z/2$. Then
for units $u$, $v$ in $\Q_2^\times$ we have 
\begin{eqnarray}
  \label{wild1}
  (u, v)_2 & = & (-1)^{\varepsilon(u) \varepsilon(v)}, \\
  \label{wild2}
  (2, u)_2 & = & (-1)^{\omega(u)}.
\end{eqnarray}
Note that these formulas completely determine $(a, b)_2$ for 
$a$, $b \in \Q_2^\times$.

Now let us return to our situation. Thus $F$ is the center of $X$ and 
contains $a_p^2 \epsilon(p)^{-1}$ for $p$ prime to $N$. 
The usual local-global exact sequence for the Brauer group of $F$ shows that the 
Brauer class of $X$ is completely determined by the Brauer classes of the $X_v$, which are in 
turn completely determined
by specifying a sign, one for each $v$. For notational convenience we write 
$X_v \sim a$ for an integer
$a$ if the sign of the Brauer class of $X_v$ is the same as
$(-1)^a$.

Recall that by a result of Momose \cite[Theorem 3.1]{Momose81} one knows $X_v \sim k$ if $v$ is 
infinite. On the other hand if
$v|p$ is a finite place of $F$ with $p \notdivides N$ we have $X_v \sim 0$ if $v$ is 
ordinary for $f$, that is if $v(a_p^2\epsilon(p)^{-1}) = 0$ (see \cite[Theorem 6]{Ribet81} for
the case $k = 2$ and \cite[Theorem 3.3.1]{Brown-Ghate03} for $k > 2$). 
The following theorem generalizes this result. 
To state it we introduce a positive integer $m_v$ for each place $v$ of $F$ of
residue characteristic $p \notdivides N$ with $a_p \neq 0$:
\begin{eqnarray}
  m_v := [F_v : \Q_p] \cdot v(a_p^2 \epsilon(p)^{-1}).
\end{eqnarray}
In the definition of $m_v$ we take the valuation $v$ which is normalized
such that $v(p)=1$.
Then we have the following theorem (it is a more precise version of Theorem~\ref{parity}):

\begin{thm} 
  \label{supersingular}
  Let $p$ be a prime such that $p \notdivides N$ and $a_p \neq 0$. 
  Let $v$ be a place of $F$ lying over $p$. 
  If $p \neq 2$ we have
  \begin{eqnarray}
    X_v \sim \begin{cases}
               0                 & \text{if $\psi_\gamma(p) = 1$ 
                                         for all $\gamma \in \Gamma_0$},            \\
              m_v                & \text{otherwise}.
             \end{cases}
  \end{eqnarray}
  If $p = 2$ then the same conclusion holds if $F = \Q$.
\end{thm}

\begin{proof} 
The proof is similar to the proof of \cite[Theorem 4.1.11]{Brown-Ghate03}
the main difference being that we use Theorem~\ref{symbols} instead of 
\cite[Theorem 4.1.3]{Brown-Ghate03} to compute $X$ locally. 

Note $[c_\epsilon]_v = 1$ if and only if the local component
$\epsilon_p$ of $\epsilon$ is even. Since $p$ is
prime to $N$, $\epsilon_p$ is in fact trivial, so that $[c_\epsilon]_v = 1$.  
It follows from Theorem~\ref{symbols} that
\begin{eqnarray}
 \label{local-symbols}
 X_v  & = & \bigotimes_{\gamma \in \Gamma_0} \> \big(z_{n_\gamma}, \> t_\gamma \big)_v.
\end{eqnarray}  
 
Since $v$ is prime to $N$ we have 
$v(t_\gamma) = 0$. First assume that $p \neq 2$. Then $v$ is prime to $2$ so
that by (\ref{tame}) we have 
$\big(z_{n_\gamma}, \> t_\gamma \big)_v = \legendre{t_\gamma}{v}^{v(z_{n_\gamma})}$.
But $\legendre{t_\gamma}{v} =  \legendre{t_\gamma}{p}^{f_v}$ since every
element of $\F_p$ has a square root over a quadratic extension of $\F_p$. 
We conclude that
\begin{eqnarray}
  \label{q-symbol}
  \big(z_{n_\gamma}, \> t_\gamma \big)_v & = & \legendre{t_\gamma}{p}^{f_v \cdot v(z_{n_\gamma})}.
\end{eqnarray}  
Thus if $\psi_\gamma(p) = \legendre{t_\gamma}{p} = 1$ for all $\gamma \in \Gamma_0$ 
then $X_v = 1$ as desired. 

Suppose on the other hand that the subset $S^-$ of the set $\{t_\gamma \divides \gamma \in \Gamma_0\}$ 
consisting of those $t_\gamma$ for which $\legendre{t_\gamma}{p} = -1$ is non-empty. 
Write the elements of $S^-$ as $t_1$, $t_2$, \ldots, $t_m$ with $m \geq 1$. 
Define distinct primes $r_j$ for $j = 0, 1, \ldots, r_{m-1}$ as follows:
set $r_0 = p$ and define $r_j$ for $j = 1, \ldots, m-1$ recursively by 
\begin{eqnarray}
  \label{def-rj-one}
  \legendre{t_i}{r_j} & = &
      (-1)^{\delta_{ij}} \cdot \legendre{t_i}{r_{j-1}} 
      \> \text{ for all } i = 1, \ldots, m, \text{ and},  \\
  \label{def-rj-two}
  \legendre{t_\gamma}{r_j}   & =  & 1 \> \text{ if } t_\gamma \notin S^-. 
\end{eqnarray}
We may and do assume that each $a_{r_j} \neq 0$. This can
be done for $j = 0$ since $a_p \neq 0$ by hypothesis.
For the other $r_j$'s we simply note that if 
$a_{r_j} = 0$ for all $r_j$ defined by the congruence 
conditions (\ref{def-rj-one}) and (\ref{def-rj-two}) then the set of 
primes $p$ for which $a_p = 0$ would have a positive density contradicting 
Serre \cite[Theorem 15]{Serre81}.

Corresponding to $t_i \in S^-$ set
\begin{eqnarray}
  \label{n_i}
  n_i = \begin{cases}
          r_{i-1} \cdot {r_i} & \text{if $1 \leq i \leq m-1$}, \\
          r_{m-1}             & \text{if $i = m$}.
        \end{cases}
\end{eqnarray}  
Clearly the $n_i$ are square free positive integers prime to the level
satisfying $a_{n_i} \neq 0$ since Fourier coefficients are multiplicative on 
distinct primes and the $r_j$ were chosen so that
each $a_{r_j} \neq 0$. Furthermore the $n_{i}$ satisfy the 
congruence conditions (\ref{n_gamma}). Indeed
suppose that $t_i$ corresponds to $\gamma \in \Gamma_0$. 
Assume first that $i < m$. Then 
\begin{eqnarray}
  \psi_\gamma(n_i) =  \psi_\gamma(r_{i-1}) \psi_\gamma(r_i) = - 1
\end{eqnarray}
since, by (\ref{def-rj-one}), $\psi_\gamma(r_{i-1})$ 
and $\psi_\gamma(r_i)$ differ by a sign. 
Similarly if $\gamma'$ corresponds to $t_j$ for $j \neq i$ then 
$\psi_{\gamma'}(r_{i-1}) = \psi_{\gamma'}(r_i)$ by (\ref{def-rj-one}) again
so that $\psi_{\gamma'}(n_i) = 1$. Finally if $\gamma'$ corresponds
to some $t_{\gamma'} \not\in S^-$ then by (\ref{def-rj-two}) $\psi_{\gamma'}(n_i) = 1$.
Now assume that $i = m$. Then for any $\gamma' \in \Gamma_0$
\begin{eqnarray}
  \psi_{\gamma'}(n_{m}) = \psi_{\gamma'}(r_{m-1}) = 
  \legendre{t_{\gamma'}}{r_{m-1}}.
\end{eqnarray}
But (\ref{def-rj-one}) shows that 
\begin{eqnarray}
  \legendre{t_i}{r_{m-1}} =
  \begin{cases}
    -     \legendre{t_i}{p} = 1   & \text{if $i \leq m-1$}, \\
    \quad \legendre{t_i}{p} = - 1 & \text{if $i = m$}. 
  \end{cases}
\end{eqnarray}
So this along with (\ref{def-rj-two}) shows that $\psi_{\gamma'}(n_m) = -1$ if and only if 
$\gamma' = \gamma$ as desired. 

We are now ready to begin computing our symbols. 
Only those $\gamma$ for which $t_i \in S^-$ contribute to the sign of 
$X_v$ in (\ref{local-symbols}) since if 
$t_\gamma \not\in S^-$ then $\legendre{t_\gamma}{v} =  \legendre{t_\gamma}{p} = 1$ and
the corresponding local symbol is trivial by (\ref{q-symbol}). 
Now for $t_i \in S^-$ we have $z_{n_i} = a_{n_i}^2 \epsilon(n_i)^{-1}$ and 
$\legendre{t_i}{p} = -1$ so that 
by (\ref{q-symbol}) we have
\begin{eqnarray}
  \big(z_{n_i}, \> t_i \big)_v & \sim & f_v \cdot v(a_{n_{i}}^2 \epsilon(n_i)^{-1}). 
\end{eqnarray}
Substituting for $n_{i}$ from (\ref{n_i}) above and multiplying over 
all $i$ in $\{1, \ldots, m \}$, there is a mod 2 telescoping effect, the result of 
which is  
\begin{eqnarray}
  X_v \sim f_v \cdot v(a_p^2 \epsilon(p)^{-1}). 
\end{eqnarray}
If we take the $v(p)=1$ normalization for $v$ then the right hand side
becomes $m_v$, proving the theorem in the case $p \neq 2$.

Now assume that $p = 2$ and that $F = \Q$. 
Write $v_2(z_{n_\gamma})$ for the power of $2$ that divides $z_{n_\gamma}$ and
define $z'_{n_\gamma}$ by $z_{n_\gamma} = 2^{v_2(z_{n_\gamma})} \cdot z'_{n_\gamma}$. 
We have
\begin{eqnarray}
  (z_{n_\gamma}, t_\gamma)_2 = (2, t_\gamma)_2^{v_2(z_{n_\gamma})} 
                                   \cdot (z'_{n_\gamma}, t_\gamma)_2.
\end{eqnarray}
Since $N$ is prime to $p=2$ by hypothesis $t_\gamma$ must be odd. One can easily check that 
$(2, t_\gamma)_2$ is equal to $(-1)^{\omega(t_\gamma)}$ by
(\ref{wild2})
which may again be easily checked to be the same as $\legendre{t_\gamma}{2}$ using
the fact $2$ splits in $\Q(\sqrt{t_\gamma})$ if and only if $t_\gamma \equiv 1 \mod 8$.
On the other hand $(z'_{n_\gamma}, t_\gamma)_2 = 1$ by (\ref{wild1}) since 
$\varepsilon(t_\gamma) \equiv 0 \mod 2$. Thus
\begin{eqnarray}
  (z_{n_\gamma}, t_\gamma)_2 = \legendre{t_\gamma}{2}^{v_2(z_{n_\gamma})}.
\end{eqnarray}
Now the argument proceeds as in the case $p \neq 2$ proving the theorem
in this case as well.
\end{proof}

\begin{remark}
  The assumption that $F = \Q$ when $p = 2$ could
  probably be removed if one had formulas for wild symbols
  other than in the case $F = \Q$.
\end{remark}

As in \cite{Brown-Ghate03} it is possible to treat the 
case $a_p = 0$ (and $p$ still prime to $N$) with minimal effort. 
The structure of $X_v$ is not determined
by the parity of $m_v$ since $m_v = \infty$. Thus the notion
of slope is not useful in measuring the ramification in this case.   
However as we now show the structure of $X_v$ is still
determined by the $v$-adic valuation of a Fourier coefficient at a prime 
$p^\dagger$, closely related to $p$. 

In fact we take $p^\dagger$ to be any prime such that 
$p p^\dagger \equiv 1 \mod N$ and such that $a_{p^\dagger} \neq 0$. 
Serre's result, quoted above, guarantees that one can always 
find such a $p^\dagger$. Set
\begin{eqnarray}
  m^\dagger_v := [F_v : \Q_p] \cdot v(a_{p^\dagger}^2 \epsilon(p^\dagger)).
\end{eqnarray} 
Theorem~\ref{supersingular} now has the following avatar when $a_p = 0$.

\begin{prop} 
  \label{a_p=0}
  Let $v$ be a place of $F$ of residue characteristic
  $p$ prime to $N$ and assume $a_p = 0$. Let $m^\dagger_v$ be as above.
  If $p \neq 2$ we have
  \begin{eqnarray}
    X_v \sim \begin{cases}
               0                 & \text{if $\psi_\gamma(p) = 1$ 
                                         for all $\gamma \in \Gamma_0$}, \\ 
              m_v^\dagger        & \text{otherwise}.
             \end{cases}
  \end{eqnarray}
  If $p = 2$ then the same conclusion holds if $F = \Q$.
\end{prop}

\begin{proof} 
  Since $p p^\dagger \equiv 1 \mod N$ we have
  \begin{eqnarray}
    \legendre{t_\gamma}{p} = \legendre{t_\gamma}{p^\dagger}
  \end{eqnarray}
  so that the proof of Theorem~\ref{supersingular} goes
  through replacing $p$ with $p^\dagger$. 
\end{proof}

Let us record the following easy consequences of the above results.

\begin{cor}
  \label{unramified-quadratic}
  Let $v$ be a place of $F$ of residue characteristic $p$ with 
  $p \notdivides 2N$. If $v$ is unramified in $E$ then $X_v$ is a
  matrix algebra over $F_v$.
\end{cor}

\begin{proof}
This is immediate from Theorem~\ref{supersingular} and 
Proposition~\ref{a_p=0} since in this case the integer
$m_v$ or $m^\dagger_v$ is necessarily even. It may also
be proved directly by studying the formula in 
Theorem~\ref{symbols}. Indeed if $v \notdivides 2N$ is a finite prime of $F$
for which $X_v$ is ramified then the normalized $v$-adic valuation
of at least one of the entries in the symbols appearing in that theorem 
must be odd. Since $t_\gamma | N$ for all $\gamma \in \Gamma$ the only 
possibility is that the normalized valuation $v(z_{q})$ must be odd for 
some prime $q$. One checks easily then that $v$ ramifies in $F(a_q)$ so that 
$v$ ramifies in $E$.   
\end{proof}

\begin{cor} 
  \label{where-ramified}
  If $X$ is ramified at $v$ then $v$ must divide either the 
  discriminant of the field $E$, or $2N$, or $\infty$.
\end{cor} 

\begin{proof} 
This is immediate from the previous corollary. 
\end{proof}

\section{Bad Places}

In the previous section we showed how the ramification of $X$ 
at the good places ($v | p \notdivides N$) is essentially
determined by the normalized slopes of $f$ at $p$. 
In this section we continue our investigation of the
ramification of $X_v$ at the bad places ($v|p|N$) that was 
begun in \cite{Brown-Ghate03}.  We recall some notation 
introduced in  that paper.

Assume now that $p|N$. Let $N_p$ be the exponent of the 
exact power of $p$ that divides $N$. 
Let $C$ denote the conductor of $\epsilon$ and let $C_p$ denote the exponent of 
the exact power of $p$ that divides $C_p$. Note $N_p \geq C_p$ and $N_p \geq 1$. 
We consider three cases:
\begin{enumerate}
  \item $N_p = C_p$: in which case $|a_p| = p^{(k-1)/2}$ (Ramified principal series),
  \item $N_p = 1$ and $C_p = 0$: in which case $a_p^2 = \epsilon(p) p^{k-2}$ (Steinberg),
  \item $N_p \neq C_p$ and $N_p \geq 2$ in which case $a_p = 0$ (Other).
\end{enumerate}

In the second case the local factor at $p$ in the automorphic representation
corresponding to $f$ is the Steinberg representation or a twist of it
by an unramified character. This case is treated 
in \cite[Theorem 3]{Ribet81} in the case $k = 2$, and 
in \cite{Brown-Ghate03} for higher weight forms  (see in particular
Theorems 3.4.6 and 3.4.8 in \cite{Brown-Ghate03}). Almost nothing
is known in the third case, where $a_p = 0$. This case
includes twists of previous cases and also cases where the
local automorphic representation is supercuspidal. Here we will be 
concerned with the first case, in which the local automorphic 
representation is in the ramified principal series. The following theorem
contains Theorems 3.4.1, 3.4.2 in \cite{Brown-Ghate03} as
special cases, and was stated without proof 
as \cite[Theorem 3.4.4]{Brown-Ghate03}.

\begin{thm}
  \label{bad-singular}
  Suppose $p|N$ with $N_p = C_p$ and let $v$ be a
  place of $F$ lying over $p$. Let $\alpha \in \Q$ be such that
  \begin{eqnarray}
  0 \leq \alpha < (k-1)/2
  \end{eqnarray}
  and $\alpha$ has odd denominator. 
  If for each place $w$ of $E$ lying over $v$ either $w(a_p) = \alpha$
  or $\bar{w}(a_p) = \alpha$ then $X_v$ is a matrix algebra over $F_v$.
\end{thm}

\begin{proof}
  Let $M_{\mathrm{crys}, v}$ denote the crystal attached to $f$ and $v$. 
  Let us recall the definition. Fix a place $w|v$ of $E$. 
  The local Galois representation $\rho_f |_{G_p} : G_p \rightarrow \GL_2(E_w)$
  is potentially crystalline. In fact if $K = \Q(\mu_{p^{r}})$ where $r = N_p = C_p$ 
  then $\rho_f |_{G_K}$ is unramified. Let 
  $D_w = D_\mathrm{st}(\rho_f |_{G_K})$ be the associated filtered module. 
  It is a free module of rank 2 over $E_w$. Set $M_{\mathrm{crys}, v} = \oplus_{w|v} D_w$.
  This has dimension $2 [E:F] [F_v:\Q_p]$ over $\Q_p$. A study of the crystals
  $D_w$ now show that the crystalline 
  Frobenius $\phi : M_{\mathrm{crys}, v} \rightarrow M_{\mathrm{crys}, v}$ 
  has characteristic polynomial
  \begin{eqnarray}
    H(x) = \prod_{w|v} \Norm_{E_w/\Q_p}\left((x-a_p)(x-\epsilon'(p) \bar{a}_p)\right) 
  \end{eqnarray}
  where $\epsilon'$ is the prime-to-$p$ part of $\epsilon$. By hypothesis
  the Newton polygon of $H(x)$ has two distinct slopes, 
  namely $\alpha$ and $k-1-\alpha$
  each occurring with equal multiplicity, say $n$.
  Let $\bar{M}_{\mathrm{crys}, v} = M_{\mathrm{crys}, v} \otimes_{\Q_p} \Q_p^\mathrm{un}$.
  It follows that $\bar{M}_{\mathrm{crys}, v} \isom C_\alpha^n \times C_{k-1-\alpha}^n$ where
  $C_\alpha$ and $C_{k-1-\alpha}$ are the simple crystals over $\Q_p^\mathrm{un}$ 
  of slopes $\alpha$ and $k-1-\alpha$. Now 
  $\dim_{\Q_p^\mathrm{un}}  C_\alpha = \dim_{\Q_p^\mathrm{un}}  C_{k-1-\alpha} = s$ 
  where $\alpha = r/s$ as a fraction in lowest terms. It follows that 
  \begin{equation}
2 [E:F] [F_v:\Q_p] = \dim_{\Q_p^\mathrm{un}}  \bar{M}_{\mathrm{crys}, v}  = 2 s n.
\end{equation} 
  Now let $V = \Hom(C_\alpha, \bar{M}_{\mathrm{crys}, v}) = \Hom(C_\alpha, C_\alpha)^n$. 
  This is a left $X_v$-module of dimension
  \begin{eqnarray}
     \dim_{F_v} V = \frac{s^2 n}{[F_v:\Q_p]} = s [E:F]
  \end{eqnarray}
  Since $s$ is odd, it follows from the representation theory
  of the algebra $X_v$ that $X_v$ must split. 
\end{proof}

\section{Tables of QM Modular Motives}

In this section we give complete tables of the endomorphism algebras of
all modular motives of small weight ($2 \leq k \leq 5$) and small 
level ($1 \leq N \leq 100$) with $F = \Q$. %(For $k = 5$ see the version of this paper on the first author's web page). 
An entry appears in the tables below only if the corresponding motive has quaternionic 
multiplication (QM for short), that is, only if the class of $X$ is non-zero in the 
Brauer group of $\Q$. 

Recall that twisting a form by a Dirichlet character
does not change the Brauer class of $X$ (see 
\cite[Proposition 3]{Ribet81} for the weight 2 case;
the proof there works in all weights). 
In the interest of conserving space we do not list those 
entries that are obtained from forms of smaller level by 
twisting. 

That $X_f$ can have non-trivial Brauer class
was discovered by Shimura: the example on page 166
of \cite{Shimura72} appears as the fifth entry in Table 1
below.
 
\begin{small}

\setlongtables
\begin{longtable}{lccccc}
\caption[Modular QM-abelian varieties of level $\leq 100$]{Modular QM-abelian varieties of level $\leq 100$}\\
\toprule
Label & ord($\epsilon$) & $E$ & Extra twists & Ramification & Slope\\[.1mm] 
\midrule
\endfirsthead
\toprule
Label & ord($\epsilon$) & $E$ & Extra twists & Ramification & Slope\\[.1mm] 
\midrule 
\endhead
\midrule
\endfoot
\bottomrule
\endlastfoot

$28A_{[ 1, 1 ]}$  & 6 &  $\Q(\sqrt{ -1}, \sqrt{ 3 })$ & $[1,0],[1,5]$ & $\begin{array}{c} 2 \\ 3\end{array}$ & $\begin{array}{c} RPS \\ 1\end{array}$ \\[.1mm]
\hline
$35A_{[ 1, 3 ]}$ & 4&  $\Q(\sqrt{ 10}, \sqrt{ -1 })$& $[0,3],[3,3]$ & $\begin{array}{c} 2 \\ 5\end{array}$ & $\begin{array}{c} 1 \\ RPS\end{array}$ \\[.1mm]
\hline
$44A_{[ 1, 5 ]}$ & 2 & $\Q(\sqrt{ 2},\sqrt{ -3 })$& $[1,5],[1,0]$ & $\begin{array}{c} 2 \\ 3\end{array}$ & $\begin{array}{c} RPS \\ 1\end{array}$ \\[.1mm]
\hline
$56B_{[ 1, 1, 3 ]}$ & 2 & $\Q(\sqrt{ -1},\sqrt{ 6 })$ & $[1,1,0],[1,1,3]$& $\begin{array}{c} 2 \\ 3\end{array}$ & $\begin{array}{c} RPS \\ 1\end{array}$ \\[.1mm]
\hline
$57A_{[ 1, 9 ]}$ & 2 & $\Q(\sqrt{ 2},\sqrt{ -5 })$& $[1,9],[0,9]$ & $\begin{array}{c} 2 \\ 5\end{array}$ & $\begin{array}{c} 1 \\ 1\end{array}$ \\[.1mm]
\hline
$60A_{[ 0, 1, 1 ]}$ & 4 & $\Q(\sqrt{ 5},\sqrt{ -1 })$ & $[0,0,3],[0,1,0]$& $\begin{array}{c} 2 \\ 5\end{array}$ & $\begin{array}{c} a_p=0 \\ RPS\end{array}$ \\[.1mm]
\hline
$63A_{[ 3, 1 ]}$ &  6 &$\Q(\sqrt{ -2},\sqrt{ 6 })$ & $[3,5],[0,5]$ & $\begin{array}{c} 2 \\ 3\end{array}$ & $\begin{array}{c} 1 \\ a_p=0\end{array}$ \\[.1mm]
\hline
$77B_{[ 3, 5 ]}$ & 2 & $\Q(\sqrt{ 10},\sqrt{ -2 })$& $[0,5],[3,5]$  & $\begin{array}{c} 2 \\ 5\end{array}$ & $\begin{array}{c} 1 \\ 1\end{array}$ \\[.1mm]
\hline
$80B_{[ 1, 0, 1 ]}$ & 4 & $\Q(\sqrt{ -1},\sqrt{ 3 })$&  $[0,0,3],[1,0,3]$ & $\begin{array}{c} 2 \\ 3\end{array}$ & $\begin{array}{c} a_p=0 \\ 1\end{array}$ \\[.1mm]
\hline
$92A_{[ 1, 11 ]}$ & 2 & $\Q(\sqrt{ -1},\sqrt{ 14 })$ & $[0,11],[1,11]$& $\begin{array}{c} 2 \\ 7\end{array}$ & $\begin{array}{c} RPS \\ 1\end{array}$ \\[.1mm]
\hline
$93D_{[ 1, 5 ]}$ & 6 & $\Q(\sqrt{ 2},\sqrt{ -3 })$ & $[1,0],[1,25]$& $\begin{array}{c} 2 \\ 3\end{array}$ & $\begin{array}{c} 1 \\ RPS\end{array}$ \\[.1mm]
\hline
$95A_{[ 1, 3 ]}$ & 12 & $\Q(\sqrt{ -1},\sqrt{ 3 })$ & $[3,15],[3,0]$ & $\begin{array}{c} 2 \\ 3\end{array}$ & $\begin{array}{c} 1 \\ 1\end{array}$ \\[.1mm]
\hline
$95B_{[ 1, 3 ]}$ & 12 & $\Q(\sqrt{ -1},\sqrt{ 3 })$ & $[3,15],[3,0]$& $\begin{array}{c} 2 \\ 3\end{array}$ & $\begin{array}{c} \infty \\ 1\end{array}$ \\[.1mm]
\end{longtable}

\setlongtables
\begin{longtable}{lccccc}
\caption[Modular QM-motives of weight 3 and level $\leq 100$]{Modular QM-motives of weight 3 and level $\leq 100$}\\
\toprule
Label & ord($\epsilon$) & $E$ & Extra twists & Ramification & Slope\\[.1mm] 
\midrule
\endfirsthead
%\toprule
%Label & ord($\epsilon$) & $E$ & Extra twists & Ramification & Slope\\[.1mm] 
%\midrule
%\endfirsthead
\toprule
Label & ord($\epsilon$) & $E$ & Extra twists & Ramification & Slope\\[.1mm] 
\midrule
\endhead
\midrule
\endfoot
\bottomrule
\endlastfoot

$9A^3_{[ 1 ]}$ & $6$ & $\Q(\sqrt{ -3 })$ & $[ 5 ]$ & $3$ & $RPS$ \\[.1mm]
\hline
$10A^3_{[0 , 1 ]}$ & $4$ & $\Q(\sqrt{ -1 })$ & $[ 3 ]$ & $2$ & $St$ \\[.1mm]
\hline
$12A^3_{[ 1, 0 ]}$ & $2$ & $\Q(\sqrt{ -3 })$ & $[ 1, 0 ]$ & $3$ & $St$ \\[.1mm]
\hline
$15A^3_{[ 1, 0 ]}$ & $2$ & $\Q(\sqrt{ -5 })$ & $[ 1, 0 ]$ & $5$ & $St$ \\[.1mm]
\hline
$18A^3_{[ 0 ,3 ]}$ & $2$ & $\Q(\sqrt{ -2 })$ & $[ 3 ]$ & $2$ & $St$ \\[.1mm]
\hline
$19B^3_{[ 9 ]}$ & $2$ & $\Q(\sqrt{ -13 })$ & $[ 9 ]$ & $13$ & $1$ \\[.1mm]
\hline
$20A^3_{[ 0, 1 ]}$ & $4$ & $\Q(\sqrt{ -1 })$ & $[ 0, 3 ]$ & $2$ & $a_p=0$ \\[.1mm]
\hline
$21A^3_{[ 0, 1 ]}$ & $6$ & $\Q(\sqrt{ -3 })$ & $[ 0, 5 ]$ & $3$ & $St$ \\[.1mm]
$21B^3_{[ 0, 1 ]}$ & $6$ & $\Q(\sqrt{ -3 })$ & $[ 0, 5 ]$ & $3$ & $St$ \\[.1mm]
$21C^3_{[ 0, 1 ]}$ & $6$ & $\Q(\sqrt{ -3 })$ & $[ 0, 5 ]$ & $3$ & $St$ \\[.1mm]
$21A^3_{[ 0, 3 ]}$ & $2$ & $\Q(\sqrt{ -3 })$ & $[ 0, 3 ]$ & $3$ & $St$ \\[.1mm]
$21B^3_{[ 1, 2 ]}$ & $6$ & $\Q(\sqrt{ -3},\sqrt{ 15 })$ & $[ 0, 4 ], [ 1, 4 ]$ & $5$ & $1$ \\[.1mm]
\hline
$22A^3_{[ 0 ,5 ]}$ & $2$ & $\Q(\sqrt{ -2 })$ & $[ 5 ]$ & $2$ & $St$ \\[.1mm]
\hline
$24A^3_{[ 0, 0, 1 ]}$ & $2$ & $\Q(\sqrt{ -2 })$ & $[ 0, 0, 1 ]$ & $2$ & $a_p=0$ \\[.1mm]
$24C^3_{[ 0, 1, 1 ]}$ & $2$ & $\Q(\sqrt{ 2},\sqrt{ -7 })$ & $[ 0, 0, 1 ], [ 0, 1, 1 ]$ & $2$ & $RPS$ \\[.1mm]
\hline
$25A^3_{[ 5 ]}$ & $4$ & $\Q(\sqrt{ -1},\sqrt{ 6 })$ & $[ 15 ], [ 5 ]$ & $3$ & $1$ \\[.1mm]
\hline
$26A^3_{[ 0 ,3 ]}$ & $4$ & $\Q(\sqrt{ -1 })$ & $[ 9 ]$ & $2$ & $St$ \\[.1mm]
\hline
$28A^3_{[ 0, 1 ]}$ & $6$ & $\Q(\sqrt{ -3 })$ & $[ 0, 5 ]$ & $3$ & $1$ \\[.1mm]
$28A^3_{[ 0, 3 ]}$ & $2$ & $\Q(\sqrt{ -6 })$ & $[ 0, 3 ]$ & $3$ & $1$ \\[.1mm]
\hline
$30A^3_{[ 0 ,1, 2 ]}$ & $2$ & $\Q(\sqrt{ 2},\sqrt{ -17 })$ & $[ 0, 2 ], [ 1, 0 ]$ & $2$ & $St$ \\[.1mm]
\hline
$31A^3_{[ 5 ]}$ & $6$ & $\Q(\sqrt{ -3 })$ & $[ 25 ]$ & $3$ & $1$ \\[.1mm]
$31A^3_{[ 15 ]}$ & $2$ & $\Q(\sqrt{ -26 })$ & $[ 15 ]$ & $13$ & $1$ \\[.1mm]
\hline
$33A^3_{[ 1, 0 ]}$ & $2$ & $\Q(\sqrt{ -11 })$ & $[ 1, 0 ]$ & $11$ & $St$ \\[.1mm]
\hline
$35A^3_{[ 0, 3 ]}$ & $2$ & $\Q(\sqrt{ -5 })$ & $[ 0, 3 ]$ & $5$ & $St$ \\[.1mm]
$35B^3_{[ 0, 3 ]}$ & $2$ & $\Q(\sqrt{ -5 })$ & $[ 0, 3 ]$ & $5$ & $St$ \\[.1mm]
$35C^3_{[ 2, 3 ]}$ & $2$ & $\Q(\sqrt{ 10},\sqrt{ -1 })$ & $[ 0, 3 ], [ 2, 3 ]$ & $5$ & $RPS$ \\[.1mm]
\hline
$36A^3_{[ 1, 2 ]}$ & $6$ & $\Q(\sqrt{ -3 })$ & $[ 1, 4 ]$ & $3$ & $RPS$ \\[.1mm]
\hline
$38A^3_{[ 0 ,9 ]}$ & $2$ & $\Q(\sqrt{ -2 })$ & $[ 9 ]$ & $2$ & $St$ \\[.1mm]
\hline
$39C^3_{[ 1, 6 ]}$ & $2$ & $\Q(\sqrt{ -35},\sqrt{ 3 })$ & $[ 1, 6 ], [ 1, 0 ]$ & $7$ & $1$ \\[.1mm]
\hline
$40A^3_{[ 0, 0, 1 ]}$ & $4$ & $\Q(\sqrt{ -1 })$ & $[ 0, 0, 3 ]$ & $2$ & $a_p=0$ \\[.1mm]
\hline
$42A^3_{[ 0 ,1, 2 ]}$ & $6$ & $\Q(\sqrt{ -2},\sqrt{ 6 })$ & $[ 0, 4 ], [ 1, 4 ]$ & $2$ & $St$ \\[.1mm]
\hline
$45A^3_{[ 0, 1 ]}$ & $4$ & $\Q(\sqrt{ -1},\sqrt{ 10 })$ & $[ 3, 0 ], [ 3, 3 ]$ & $5$ & $RPS$ \\[.1mm]
$45A^3_{[ 3, 2 ]}$ & $2$ & $\Q(\sqrt{ 7},\sqrt{ -2 })$ & $[ 0, 2 ], [ 3, 2 ]$ & $7$ & $1$ \\[.1mm]
\hline
$47A^3_{[ 23 ]}$ & $2$ & $\Q(\sqrt{ -78 })$ & $[ 23 ]$ & $13$ & $1$ \\[.1mm]
\hline
$48A^3_{[ 1, 0, 0 ]}$ & $2$ & $\Q(\sqrt{ -3 })$ & $[ 1, 0, 0 ]$ & $3$ & $St$ \\[.1mm]
\hline
$50A^3_{[ 0 ,5 ]}$ & $4$ & $\Q(\sqrt{ -1 })$ & $[ 15 ]$ & $2$ & $St$ \\[.1mm]
\hline
$54A^3_{[ 0 ,9 ]}$ & $2$ & $\Q(\sqrt{ -2 })$ & $[ 9 ]$ & $2$ & $St$ \\[.1mm]
\hline
$55D^3_{[ 2, 5 ]}$ & $2$ & $\Q(\sqrt{ -21},\sqrt{ 5 })$ & $[ 2, 0 ], [ 2, 5 ]$ & $3$ & $1$ \\[.1mm]
\hline
$56A^3_{[ 1, 1, 2 ]}$ & $6$ & $\Q(\sqrt{ -3 })$ & $[ 1, 1, 4 ]$ & $2$ & $RPS$ \\[.1mm]
\hline
$57A^3_{[ 0, 9 ]}$ & $2$ & $\Q(\sqrt{ -3 })$ & $[ 0, 9 ]$ & $3$ & $St$ \\[.1mm]
\hline
$60A^3_{[ 0, 1, 0 ]}$ & $2$ & $\Q(\sqrt{ -5 })$ & $[ 0, 1, 0 ]$ & $5$ & $St$ \\[.1mm]
$60A^3_{[ 0, 1, 2 ]}$ & $2$ & $\Q(\sqrt{ -1},\sqrt{ 5 })$ & $[ 0, 0, 2 ], [ 0, 1, 0 ]$ & $5$ & $RPS$ \\[.1mm]
$60A^3_{[ 1, 0, 2 ]}$ & $2$ & $\Q(\sqrt{ -1},\sqrt{ 3 })$ & $[ 1, 0, 2 ], [ 1, 0, 0 ]$ & $3$ & $St$ \\[.1mm]
\hline
$63E^3_{[ 0, 1 ]}$ & $6$ & $\Q(\sqrt{ -3},\sqrt{ 13 })$ & $[ 3, 5 ], [ 0, 5 ]$ & $3$ & $a_p=0$ \\[.1mm]
\hline
$64A^3_{[ 1, 8 ]}$ & $2$ & $\Q(\sqrt{ -1},\sqrt{ 3 })$ & $[ 1, 8 ], [ 1, 0 ]$ & $3$ & $1$ \\[.1mm]
\hline
$72A^3_{[ 0, 0, 3 ]}$ & $2$ & $\Q(\sqrt{ -2 })$ & $[ 0, 0, 3 ]$ & $2$ & $a_p=0$ \\[.1mm]
$72C^3_{[ 1, 1, 0 ]}$ & $2$ & $\Q(\sqrt{ 10},\sqrt{ -6 })$ & $[ 0, 0, 3 ], [ 1, 1, 0 ]$ & $2$ & $RPS$ \\[.1mm]
\hline
$74A^3_{[ 0 ,9 ]}$ & $4$ & $\Q(\sqrt{ -1 })$ & $[ 27 ]$ & $2$ & $St$ \\[.1mm]
$74B^3_{[ 0 ,9 ]}$ & $4$ & $\Q(\sqrt{ -1 })$ & $[ 27 ]$ & $2$ & $St$ \\[.1mm]
$74C^3_{[ 0 ,9 ]}$ & $4$ & $\Q(\sqrt{ -1 })$ & $[ 27 ]$ & $2$ & $St$ \\[.1mm]
\hline
$75A^3_{[ 0, 5 ]}$ & $4$ & $\Q(\sqrt{ -1},\sqrt{ 6 })$ & $[ 0, 15 ], [ 0, 5 ]$ & $3$ & $St$ \\[.1mm]
$75B^3_{[ 0, 5 ]}$ & $4$ & $\Q(\sqrt{ -1},\sqrt{ 6 })$ & $[ 0, 15 ], [ 0, 5 ]$ & $3$ & $St$ \\[.1mm]
$75C^3_{[ 1, 0 ]}$ & $2$ & $\Q(\sqrt{ -11 })$ & $[ 1, 0 ]$ & $11$ & $1$ \\[.1mm]
\hline
$76A^3_{[ 0, 9 ]}$ & $2$ & $\Q(\sqrt{ -29 })$ & $[ 0, 9 ]$ & $29$ & $1$ \\[.1mm]
\hline
$77A^3_{[ 0, 5 ]}$ & $2$ & $\Q(\sqrt{ -7 })$ & $[ 0, 5 ]$ & $7$ & $St$ \\[.1mm]
$77A^3_{[ 2, 5 ]}$ & $6$ & $\Q(\sqrt{ -3},\sqrt{ 21 })$ & $[ 4, 5 ], [ 2, 0 ]$ & $7$ & $RPS$ \\[.1mm]
\hline
$78A^3_{[ 0 ,1, 6 ]}$ & $2$ & $\Q(\sqrt{ 2},\sqrt{ -5 })$ & $[ 0, 6 ], [ 1, 0 ]$ & $2$ & $St$ \\[.1mm]
$78B^3_{[ 0 ,1, 6 ]}$ & $2$ & $\Q(\sqrt{ 2},\sqrt{ -5 })$ & $[ 0, 6 ], [ 1, 0 ]$ & $2$ & $St$ \\[.1mm]
\hline
$80B^3_{[ 1, 0, 2 ]}$ & $2$ & $\Q(\sqrt{ 2},\sqrt{ -3 })$ & $[ 0, 0, 2 ], [ 1, 0, 0 ]$ & $2$ & $a_p=0$ \\[.1mm]
\hline
$84A^3_{[ 0, 0, 1 ]}$ & $6$ & $\Q(\sqrt{ -3 })$ & $[ 0, 0, 5 ]$ & $3$ & $St$ \\[.1mm]
$84A^3_{[ 0, 0, 3 ]}$ & $2$ & $\Q(\sqrt{ -3 })$ & $[ 0, 0, 3 ]$ & $3$ & $St$ \\[.1mm]
$84B^3_{[ 0, 1, 2 ]}$ & $6$ & $\Q(\sqrt{ -3},\sqrt{ 15 })$ & $[ 0, 0, 2 ], [ 0, 1, 4 ]$ & $5$ & $1$ \\[.1mm]
$84C^3_{[ 0, 1, 2 ]}$ & $6$ & $\Q(\sqrt{ -3},\sqrt{ 105 })$ & $[ 0, 0, 4 ], [ 0, 1, 4 ]$ & $5$ & $1$ \\[.1mm]
$84A^3_{[ 1, 0, 2 ]}$ & $6$ & $\Q(\sqrt{ -3 })$ & $[ 1, 0, 4 ]$ & $3$ & $St$ \\[.1mm]
\hline
$86A^3_{[ 0 ,21 ]}$ & $2$ & $\Q(\sqrt{ -2 })$ & $[ 21 ]$ & $2$ & $St$ \\[.1mm]
\hline
$90A^3_{[ 0 ,0, 1 ]}$ & $4$ & $\Q(\sqrt{ -1 })$ & $[ 0, 3 ]$ & $2$ & $St$ \\[.1mm]
$90B^3_{[ 0 ,0, 1 ]}$ & $4$ & $\Q(\sqrt{ -1 })$ & $[ 0, 3 ]$ & $2$ & $St$ \\[.1mm]
$90A^3_{[ 0 ,3, 2 ]}$ & $2$ & $\Q(\sqrt{ 2},\sqrt{ -1 })$ & $[ 0, 2 ], [ 3, 2 ]$ & $2$ & $St$ \\[.1mm]
\hline
$91A^3_{[ 3, 4 ]}$ & $6$ & $\Q(\sqrt{ -3},\sqrt{ 39 })$ & $[ 0, 4 ], [ 3, 0 ]$ & $13$ & $RPS$ \\[.1mm]
$91C^3_{[ 3, 6 ]}$ & $2$ & $\Q(\sqrt{ 26},\sqrt{ -1 })$ & $[ 3, 0 ], [ 3, 6 ]$ & $13$ & $RPS$ \\[.1mm]
\hline
$93A^3_{[ 0, 15 ]}$ & $2$ & $\Q(\sqrt{ -3 })$ & $[ 0, 15 ]$ & $3$ & $St$ \\[.1mm]
\hline
$96A^3_{[ 0, 0, 1 ]}$ & $2$ & $\Q(\sqrt{ -2},\sqrt{ 3 })$ & $[ 0, 0, 1 ], [ 1, 0, 0 ]$ & $3$ & $RPS$ \\[.1mm]
$96B^3_{[ 0, 0, 1 ]}$ & $2$ & $\Q(\sqrt{ 7},\sqrt{ -2 })$ & $[ 0, 0, 1 ], [ 1, 0, 1 ]$ & $7$ & $1$ \\[.1mm]
\hline
$99C^3_{[ 0, 5 ]}$ & $2$ & $\Q(\sqrt{ -138},\sqrt{ -3 })$ & $[ 3, 0 ], [ 0, 5 ]$ & $23$ & $1$ \\[.1mm]
\hline
$100B^3_{[ 0, 5 ]}$ & $4$ & $\Q(\sqrt{ -1},\sqrt{ 6 })$ & $[ 0, 15 ], [ 0, 5 ]$ & $3$ & $3$ \\[.1mm]
\end{longtable}

\setlongtables
\begin{longtable}{lccccc}
\caption[Modular QM-motives of weight 4 and level $\leq 100$]{Modular QM-motives of weight 4 and level $\leq 100$}\\
\toprule
Label & ord($\epsilon$) & $E$ & Extra twists & Ramification & Slope\\[.1mm] 
\midrule
\endfirsthead
\toprule
Label & ord($\epsilon$) & $E$ & Extra twists & Ramification & Slope\\[.1mm] 
\endhead
\endfoot
\bottomrule
\endlastfoot
$12A^4_{[ 1, 1 ]}$ & $2$ & $\Q(\sqrt{ -5},\sqrt{ 3 })$ & $[ 0, 1 ], [ 1, 0 ]$ & $\begin{array}{c} 3 \\ 5 \end{array}$ & $\begin{array}{c} RPS \\ 1\end{array}$ \\[.1mm]
\hline
$21B^4_{[ 1, 3 ]}$ & $2$ & $\Q(\sqrt{ -6},\sqrt{ 102 })$ & $[ 1, 3 ], [ 0, 3 ]$ & $\begin{array}{c} 3 \\ 17 \end{array}$ & $\begin{array}{c} RPS \\ 1\end{array}$ \\[.1mm]
\hline
$27C^4_{[ 0 ]}$ & $1$ & $\Q(\sqrt{ 2 })$ & $[ 9 ]$ & $\begin{array}{c} 2 \\ 3 \end{array}$ & $\begin{array}{c} 1 \\ a_p=0\end{array}$ \\[.1mm]
\hline
$35A^4_{[ 1, 3 ]}$ & $4$ & $\Q(\sqrt{ 5},\sqrt{ -1 })$ & $[ 0, 3 ], [ 3, 0 ]$ & $\begin{array}{c} 2 \\ 5 \end{array}$ & $\begin{array}{c} 3 \\ RPS\end{array}$ \\[.1mm]
\hline
$36B^4_{[ 1, 3 ]}$ & $2$ & $\Q(\sqrt{ 30},\sqrt{ -2 })$ & $[ 0, 3 ], [ 1, 0 ]$ & $\begin{array}{c} 2 \\ 3 \end{array}$ & $\begin{array}{c} RPS \\ a_p=0\end{array}$ \\[.1mm]
\hline
$48B^4_{[ 1, 0, 1 ]}$ & $2$ & $\Q(\sqrt{ -2},\sqrt{ 6 })$ & $[ 1, 0, 0 ], [ 1, 0, 1 ]$ & $\begin{array}{c} 2 \\ 3 \end{array}$ & $\begin{array}{c} a_p=0 \\ RPS\end{array}$ \\[.1mm]
\hline
$56B^4_{[ 1, 1, 3 ]}$ & $2$ & $\Q(\sqrt{ -3},\sqrt{ 21 })$ & $[ 0, 0, 3 ], [ 1, 1, 3 ]$ & $\begin{array}{c} 3 \\ 7 \end{array}$ & $\begin{array}{c} 1 \\ RPS\end{array}$ \\[.1mm]
\hline
$57B^4_{[ 1, 9 ]}$ & $2$ & $\Q(\sqrt{ 17},\sqrt{ -10 })$ & $[ 1, 9 ], [ 0, 9 ]$ & $\begin{array}{c} 5 \\ 17 \end{array}$ & $\begin{array}{c} 1 \\ 1\end{array}$ \\[.1mm]
\hline
$63B^4_{[ 3, 3 ]}$ & $2$ & $\Q(\sqrt{ -222},\sqrt{ -2 })$ & $[ 3, 0 ], [ 0, 3 ]$ & $\begin{array}{c} 2 \\ 3 \end{array}$ & $\begin{array}{c} 3 \\ a_p=0\end{array}$ \\[.1mm]
\hline
$72C^4_{[ 0, 1, 0 ]}$ & $2$ & $\Q(\sqrt{ 22},\sqrt{ -10 })$ & $[ 0, 0, 3 ], [ 0, 1, 0 ]$ & $\begin{array}{c} 2 \\ 5 \end{array}$ & $\begin{array}{c} RPS \\ 1\end{array}$ \\[.1mm]
\hline
$80B^4_{[ 1, 0, 1 ]}$ & $4$ & $\Q(\sqrt{ -1},\sqrt{ 35 })$ & $[ 1, 0, 0 ], [ 1, 0, 3 ]$ & $\begin{array}{c} 2 \\ 7 \end{array}$ & $\begin{array}{c} a_p=0 \\ 1\end{array}$ \\[.1mm]
\hline
$100D^4_{[ 1, 5 ]}$ & $4$ & $\Q(\sqrt{ 11},\sqrt{ -1},\sqrt{ 5 })$ & $[ 1, 10 ], [ 0, 5 ], [ 1, 5 ]$ & $\begin{array}{c} 11 \\ 2 \end{array}$ & $\begin{array}{c} 1 \\ RPS\end{array}$ \\[.1mm]

\end{longtable}

\setlongtables
\begin{longtable}{lccccc}
\caption[Modular QM-motives of weight 5 and level $\leq 100$]{Modular QM-motives of weight 5 and level $\leq 100$}\\
\toprule
Label & ord($\epsilon$) & $E$ & Extra twists & Ramification & Slope\\[.1mm] 
\midrule
\endfirsthead
\toprule
Label & ord($\epsilon$) & $E$ & Extra twists & Ramification & Slope\\[.1mm] 
\midrule
\endhead
\endfoot
\bottomrule
\endlastfoot
$5A^5_{[ 1 ]}$ & $4$ & $\Q(\sqrt{ -1 })$ & $[ 3 ]$ & $2$ & $1$ \\[.1mm]
\hline
$6A^5_{[ 0 ,1 ]}$ & $2$ & $\Q(\sqrt{ -2 })$ & $[ 1 ]$ & $2$ & $St$ \\[.1mm]
\hline
$8B^5_{[ 1, 1 ]}$ & $2$ & $\Q(\sqrt{ -15 })$ & $[ 1, 1 ]$ & $5$ & $1$ \\[.1mm]
\hline
$9A^5_{[ 3 ]}$ & $2$ & $\Q(\sqrt{ -2 })$ & $[ 3 ]$ & $2$ & $1$ \\[.1mm]
\hline
$10A^5_{[0 , 1 ]}$ & $4$ & $\Q(\sqrt{ -1 })$ & $[ 3 ]$ & $2$ & $St$ \\[.1mm]
$10B^5_{[ 0 ,1 ]}$ & $4$ & $\Q(\sqrt{ -1 })$ & $[ 3 ]$ & $2$ & $St$ \\[.1mm]
\hline
$11B^5_{[ 5 ]}$ & $2$ & $\Q(\sqrt{ -30 })$ & $[ 5 ]$ & $2$ & $1$ \\[.1mm]
\hline
$15C^5_{[ 1, 2 ]}$ & $2$ & $\Q(\sqrt{ 10},\sqrt{ -26 })$ & $[ 0, 2 ], [ 1, 2 ]$ & $\begin{array}{c} 2 \\ 13 \\ 5\end{array}$ & $\begin{array}{c} 1 \\ 1\\ RPS\end{array}$ \\[.1mm]
\hline
$16A^5_{[ 1, 0 ]}$ & $2$ & $\Q(\sqrt{ -3 })$ & $[ 1, 0 ]$ & $3$ & $1$ \\[.1mm]
\hline
$21A^5_{[ 0, 1 ]}$ & $6$ & $\Q(\sqrt{ -3 })$ & $[ 0, 5 ]$ & $3$ & $St$ \\[.1mm]
$21B^5_{[ 0, 1 ]}$ & $6$ & $\Q(\sqrt{ -3 })$ & $[ 0, 5 ]$ & $3$ & $St$ \\[.1mm]
\hline
$25B^5_{[ 5 ]}$ & $4$ & $\Q(\sqrt{ -1},\sqrt{ 6 })$ & $[ 15 ], [ 5 ]$ & $3$ & $1$ \\[.1mm]
$25C^5_{[ 5 ]}$ & $4$ & $\Q(\sqrt{ -1},\sqrt{ 21 })$ & $[ 15 ], [ 10 ]$ & $\begin{array}{c} 2 \\ 3 \\ 7\end{array}$ & $\begin{array}{c} 1 \\ 1\\ 1\end{array}$ \\[.1mm]
\hline
$27B^5_{[ 9 ]}$ & $2$ & $\Q(\sqrt{ -6 })$ & $[ 9 ]$ & $2$ & $1$ \\[.1mm]
\hline
$28A^5_{[ 0, 3 ]}$ & $2$ & $\Q(\sqrt{ -3 })$ & $[ 0, 3 ]$ & $3$ & $1$ \\[.1mm]
\hline
$35C^5_{[ 2, 3 ]}$ & $2$ & $\Q(\sqrt{ -6 })$ & $[ 2, 3 ]$ & $2$ & $1$ \\[.1mm]
\hline
$36A^5_{[ 0, 3 ]}$ & $2$ & $\Q(\sqrt{ -2 })$ & $[ 0, 3 ]$ & $2$ & $a_p=0$ \\[.1mm]
$36C^5_{[ 1, 0 ]}$ & $2$ & $\Q(\sqrt{ 7},\sqrt{ -1 })$ & $[ 1, 3 ], [ 1, 0 ]$ & $7$ & $1$ \\[.1mm]
\hline
$40A^5_{[ 0, 0, 1 ]}$ & $4$ & $\Q(\sqrt{ -1 })$ & $[ 0, 0, 3 ]$ & $2$ & $a_p=0$ \\[.1mm]
\hline
$44B^5_{[ 0, 5 ]}$ & $2$ & $\Q(\sqrt{ -206 })$ & $[ 0, 5 ]$ & $2$ & $a_p=0$ \\[.1mm]
\hline
$45A^5_{[ 0, 1 ]}$ & $4$ & $\Q(\sqrt{ -1 })$ & $[ 0, 3 ]$ & $2$ & $1$ \\[.1mm]
$45B^5_{[ 0, 1 ]}$ & $4$ & $\Q(\sqrt{ -1 })$ & $[ 0, 3 ]$ & $2$ & $1$ \\[.1mm]
$45D^5_{[ 0, 1 ]}$ & $4$ & $\Q(\sqrt{ -1},\sqrt{ 10 })$ & $[ 3, 0 ], [ 3, 3 ]$ & $5$ & $RPS$ \\[.1mm]
\hline
$48A^5_{[ 1, 0, 0 ]}$ & $2$ & $\Q(\sqrt{ -3 })$ & $[ 1, 0, 0 ]$ & $3$ & $St$ \\[.1mm]
$48B^5_{[ 1, 0, 0 ]}$ & $2$ & $\Q(\sqrt{ -3 })$ & $[ 1, 0, 0 ]$ & $3$ & $St$ \\[.1mm]
\hline
$54A^5_{[0 , 9 ]}$ & $2$ & $\Q(\sqrt{ -2 })$ & $[ 9 ]$ & $2$ & $St$ \\[.1mm]
\hline
$56A^5_{[ 1, 1, 0 ]}$ & $2$ & $\Q(\sqrt{ -7 })$ & $[ 1, 1, 0 ]$ & $7$ & $St$ \\[.1mm]
\hline
$60A^5_{[ 0, 1, 0 ]}$ & $2$ & $\Q(\sqrt{ -5 })$ & $[ 0, 1, 0 ]$ & $5$ & $St$ \\[.1mm]
\hline
$63C^5_{[ 0, 3 ]}$ & $2$ & $\Q(\sqrt{ 10},\sqrt{ -106 })$ & $[ 3, 0 ], [ 3, 3 ]$ & $2$ & $1$ \\[.1mm]
\hline
$64A^5_{[ 1, 8 ]}$ & $2$ & $\Q(\sqrt{ -1},\sqrt{ 3 })$ & $[ 1, 8 ], [ 1, 0 ]$ & $3$ & $1$ \\[.1mm]
$64B^5_{[ 1, 8 ]}$ & $2$ & $\Q(\sqrt{ 51},\sqrt{ -1 })$ & $[ 1, 8 ], [ 0, 8 ]$ & $3$ & $1$ \\[.1mm]
\hline
$72A^5_{[ 0, 0, 3 ]}$ & $2$ & $\Q(\sqrt{ -2 })$ & $[ 0, 0, 3 ]$ & $2$ & $a_p=0$ \\[.1mm]
$72B^5_{[ 0, 0, 3 ]}$ & $2$ & $\Q(\sqrt{ -2 })$ & $[ 0, 0, 3 ]$ & $2$ & $a_p=0$ \\[.1mm]
$72B^5_{[ 1, 1, 0 ]}$ & $2$ & $\Q(\sqrt{ -15 })$ & $[ 1, 1, 0 ]$ & $5$ & $1$ \\[.1mm]
\hline
$75B^5_{[ 0, 5 ]}$ & $4$ & $\Q(\sqrt{ -1},\sqrt{ 6 })$ & $[ 0, 15 ], [ 0, 5 ]$ & $3$ & $St$ \\[.1mm]
$75C^5_{[ 0, 5 ]}$ & $4$ & $\Q(\sqrt{ -1},\sqrt{ 6 })$ & $[ 0, 15 ], [ 0, 5 ]$ & $3$ & $St$ \\[.1mm]
$75C^5_{[ 1, 0 ]}$ & $2$ & $\Q(\sqrt{ -14 })$ & $[ 1, 0 ]$ & $2$ & $1$ \\[.1mm]
$75D^5_{[ 1, 0 ]}$ & $2$ & $\Q(\sqrt{ -35 })$ & $[ 1, 0 ]$ & $5$ & $a_p=0$ \\[.1mm]
\hline
$80D^5_{[ 0, 0, 1 ]}$ & $4$ & $\Q(\sqrt{ -1 })$ & $[ 0, 0, 3 ]$ & $2$ & $a_p=0$ \\[.1mm]
\hline
$81D^5_{[ 9 ]}$ & $6$ & $\Q(\sqrt{ 2},\sqrt{ -3 })$ & $[ 27 ], [ 45 ]$ & $2$ & $1$ \\[.1mm]
\hline
$90A^5_{[ 0 ,0, 1 ]}$ & $4$ & $\Q(\sqrt{ -1 })$ & $[ 0, 3 ]$ & $2$ & $St$ \\[.1mm]
$90B^5_{[ 0 ,0, 1 ]}$ & $4$ & $\Q(\sqrt{ -1 })$ & $[ 0, 3 ]$ & $2$ & $St$ \\[.1mm]
\hline
$99B^5_{[ 0, 5 ]}$ & $2$ & $\Q(\sqrt{ -30 })$ & $[ 0, 5 ]$ & $2$ & $1$ \\[.1mm]
\hline
$100A^5_{[ 0, 5 ]}$ & $4$ & $\Q(\sqrt{ -1},\sqrt{ 6 })$ & $[ 0, 15 ], [ 0, 5 ]$ & $3$ & $1$ \\[.1mm]
$100B^5_{[ 0, 5 ]}$ & $4$ & $\Q(\sqrt{ -1},\sqrt{ -69 })$ & $[ 0, 5 ], [ 0, 10 ]$ & 
$\begin{array}{c} 2 \\ 3 \\ 23\end{array}$ & $\begin{array}{c} a_p=0 \\ 1 \\ 1 \end{array}$ \\[.1mm]

\end{longtable}

\end{small}

Let us describe how the above tables are labeled.
The format is similar to that used in the tables in 
the Appendix of \cite{Baker-Gonzalez-Gonzalez-Poonen03}. 
The first column contains Galois conjugacy classes of 
primitive forms of given level, weight and nebentypus.
The ordering we use to list these forms is described by a 
function which maps a  primitive form $f$ 
to a label of the form $NX^k_\eps$ (for example $19B^3_{[9]}$),
where $N$ is the level of $f$, $X$ is a letter or string of letters 
in $\{A,B,\dots,Z,AA,BB,\dots\}$, $\eps$ is an encoding of the nebentypus of $f$ 
(described in more detail below) and $k$ is the weight of $f$. 
When $k =2$ we omit the superscript 2. 

To construct $X$ assume that $k$, $N$ and $\eps$ are fixed.
To $f = \sum a_n q^n$ associate the infinite sequence of integers 
${\mathbf t}_f = (\Tr_{E/\Q} a_1, \Tr_{E/\Q} a_2, \dots)$.
Choose $X \in \{A,B,\dots,Z,AA,BB,\dots\}$
according to the position of ${\mathbf t}_f$
in the set $\{\, {\mathbf t}_g : \text{ $g$ primiti-}$ $\text{ve of weight $k$, level $N$ and 
                                        nebentypus $\eps$} \,\}$ 
sorted in increasing dictionary order. 
Notice that ${\mathbf t}_f$ determines the Galois conjugacy class of $f$.
The above ordering was introduced by J. Cremona in the case of trivial nebentypus and 
weight $2$ and by W. Stein in the  general situation. 

The encoding of the nebentypus $\eps:(\Z/N)^*\rightarrow \C^*$ 
is done as follows. 
Let $N=\prod{p_n^{\alpha_n}}$ be the prime-ordered factorization of $N$.
Then for each $p_n$ there exists a unique Dirichlet character 
$\eps_{p_n}:(\Z/p_n^{\alpha_n})^ \times \rightarrow \C^*$ such that $\eps=\prod{\eps_{p_n}}$. 
Fix $p = p_n$ momentarily and write $\eps_p$ for $\eps_{p_n}$.
If $p$ is odd let $g_p$ be the smallest positive integer 
that generates $(\Z/p^{\alpha})^\times$,
and if $p=2$ and $\alpha \le 2$,
let $g_p=-1$. In the above cases $\eps_{p}$ is determined by the integer 
$e_p \in [0,\varphi(p^\alpha))$
such that $\eps_p(g_p)=e^{2\pi i e_p/\varphi(p^\alpha)}$.
If $p=2$ and $\alpha > 2$,
then $(\Z/2^\alpha)^\times \isom \Z/2 \times \Z/2^{\alpha-2}$ where the first factor
is generated by $-1$ and the second factor by 5. Thus in this case $\eps_2$ is determined 
by a pair of integers $e_2' \in [0,2)$, $e_2'' \in [0, 2^{\alpha-2})$ such that 
$\eps_2(-1)=e^{2\pi i e_2'/2}$
and $\eps_2(5)=e^{2\pi i e_2''/2^{\alpha-2}}$.
We denote the pair $e_2', e_2''$ by $e_2$.
Finally we denote $\eps$ by $[e_{p_n}\,:\,p_n|N]$.

The middle columns are as follows: column 2 contains the order of $\eps$, column 3
contains the Hecke field $E$ (recall that $F = \Q)$, column 4 lists a generating
set for the extra twists (encoded in a similar manner as described above for $\eps$),
and column 5 lists the primes where the endomorphism algebra $X$ ramifies.

The last column lists the numbers $m_v$ (normalized slope) if $p$ is prime to $N$ and 
$a_p \neq 0$. If a (finite) integer occurs in this column it is always ODD as predicted
by the main result of this paper (Theorem~\ref{parity}). 
If $a_p = 0$ and $p$ is still prime to $N$ then
$m_v = \infty$ and the ramification is controlled by Proposition~\ref{a_p=0}.
On the other hand if some $p | N$ is a prime of ramification then 
we give some further information as follows. 
Recall that $C$ denotes the conductor of $\eps$, and $N_p$ and $C_p$ denote the exponent of 
$p$ of the exact power of $p$ diving $N$ and $C$ respectively. 
If $N_p = C_p$ we write $RPS$ for ramified principal series. 
If $N_p = 1$ and $C_p = 0$ then we write $St$ for Steinberg. Finally $a_p$ vanishes 
in the remaining cases $N_p \neq C_p$. These include the
cases which are twists of previous cases, in which case the ramification can
be sometimes explained, but also includes the cases where the local representation 
is supercuspidal. In either case we simply write $a_p = 0$. It is 
worth mentioning that while $a_p$ can vanish both when $p \notdivides N$ and $p | N$,
in the former case it only occurs once (see $95B_{[1,3]}$ below) within the scope of the 
tables.

\vskip .2cm

{\noindent \bf Corrections to \cite{Brown-Ghate03}:}
We take this opportunity to correct some errors in \cite{Brown-Ghate03}. 
The last two errors do not occur in the electronic version of the article.
\begin{itemize}
  \item[(i)] page 1655, line 8: $N_p \geq C_p \geq 1$ should be $N_p \geq C_p \geq 0$
  \item[(ii)] page 1669, line 12: $\frac{1}{2}(u^2-1)$ should be $(u^2-1)/8$
  \item[(iii)] page 1670, last few lines: `supersingular primes' should read `primes $p$ for which
                                   $a_p = 0$', and the last phrase should read 
                                   `for non-CM forms the
                                   density of primes $p$ for which $a_p = 0$ is $0$' 
  \item[(iv)] page 1672, line 9 from the bottom: `ordinary primes' should read 
                                   `primes $p$ for which $a_p = 0$'.
\end{itemize}

\vskip .2 cm

{\noindent \bf Acknowledgements:} The first author would like to thank
N. Fakhruddin, M. Nori, and V. Srinivas for useful discussions
concerning the material in Section 3. He would also like to thank the
third author for arranging visits to Barcelona in 2002 and 2004. 
The second author held the position of {\it Contratado Doctor 
C.A.M. (02/0294/2002)} during the preparation of this work. The third
author is supported by research projects BFM-2003-0678-C02-01 and
2002SGR 00148. The tables in 
this paper were prepared using W. Stein's Magma modular forms package.

%\bibliographystyle{plain}
%\bibliography{slopes}

\end{document}